\documentclass[a4paper,11pt,reqno]{amsart}

\usepackage{amsthm,amssymb,amsmath,amsopn}
\usepackage{xypic}

\newtheorem{theorem}{Theorem}[section]
\newtheorem{lemma}[theorem]{Lemma}
\newtheorem{corollary}[theorem]{Corollary}
\newtheorem{proposition}[theorem]{Proposition}

\newtheoremstyle{definition}{}{}{}{}{}{}{2ex}{\bf #1.}
\theoremstyle{definition}
\newtheorem*{definition}{Definition}
\newtheorem*{claim}{Claim}
\newtheorem*{example}{Example}
\newtheorem*{remark}{Remark}
\newtheorem*{conjecture}{Conjecture}

\newcommand{\F}{\mathbf{F}}

\newcommand{\Z}{\mathbf{Z}}
\newcommand{\N}{\mathbf{N}}
\newcommand{\R}{\mathbf{R}}
\newcommand{\Q}{\mathbf{R}}

\newcommand{\Mat}{\mathrm{Mat}}
\newcommand{\PGL}{\mathrm{PGL}}
\newcommand{\GL}{\mathrm{GL}}
\newcommand{\SL}{\mathrm{SL}}

\DeclareMathOperator{\ty}{ty}
\DeclareMathOperator{\cyc}{cyc}
\DeclareMathOperator{\diag}{diag}
\DeclareMathOperator{\Cent}{Cent}
\DeclareMathOperator{\chr}{char}

\DeclareMathOperator{\im}{im}

\DeclareMathOperator{\height}{ht}

\DeclareMathOperator{\rank}{rank}
\DeclareMathOperator{\nullity}{null}

\DeclareMathOperator{\hcf}{hcf}
\DeclareMathOperator{\lcm}{lcm}

\begin{document}

\begin{abstract}
This paper addresses various questions about pairs of similarity classes of matrices which contain commuting elements.
In the case of matrices over finite fields, we show that the problem of determining such pairs reduces to a question
about nilpotent classes; this reduction
makes use of class types in the sense of Steinberg and Green.
We investigate the set of scalars that arise as determinants of elements of the
centralizer algebra of a matrix, providing a complete description of this set in terms of the class type of the matrix.

Several results are established concerning the commuting of
nilpotent classes. Classes which are represented in the centralizer of every nilpotent matrix are classified---this result holds over any field.
Nilpotent classes are parametrized by partitions; we find pairs of partitions whose corresponding nilpotent classes commute over
some finite fields, but not over others. We conclude by classifying all pairs of classes, parametrized by two-part partitions, that commute. Our
results on nilpotent classes complement work of Ko\v{s}ir and Oblak.
\end{abstract}

\title[Types and classes of commuting matrices]
{On types and classes of commuting matrices over finite fields}

%\classno{15A27 (primary), 15A15, 15A21 (secondary)}

\author{John R. Britnell and Mark Wildon}
\date{15 December 2009}

%\extraline{The research was supported in part by the Heilbronn
%Institute for Mathematical Research.}

\maketitle

\section{General introduction}\label{sect:introduction}

Let $\F_q$ be a finite field, and let $C$ and $D$
be classes of similar matrices in $\Mat_n(\F_q)$. We say that
$C$ and $D$ \emph{commute} if there exist commuting
matrices $X$ and $Y$ such that $X \in C$ and~$Y \in D$. In this paper we
are concerned with the problem of deciding
which similarity classes commute.

A matrix is determined up to similarity by its rational
canonical form. This however is usually too sharp a tool
for our purposes, and many of our results are instead
stated in terms of the \emph{class type} of a matrix. This notion,
which seems first to have appeared in the work of Steinberg \cite{Steinberg},
is important in
Green's influential paper \cite{Green} on the characters of finite general linear groups.
Lemma 2.1 of that paper implies that the type of a matrix determines its centralizer up to isomorphism; this fact
is also implied by our Theorem \ref{thm:conjugatecentralizer}, which says that two matrices with the same class type have
conjugate centralizers.

The main body of this paper is divided into three sections.
In \S 2 we develop
a theory of commuting class types; the results of this section
reduce the general problem of determining commuting classes
to the case of nilpotent classes.
A key step in this reduction is
Theorem~\ref{thm:typecommute}, which states that if similarity classes $C$ and $D$ commute, then any class of the
type of $C$ commutes with any class of the type of $D$.

Relationships between class types and determinants are discussed in \S 3. We provide a complete account of those scalars which appear as determinants in the centralizer of a
matrix of a given type; this result, stated as Theorem \ref{thm:determinants}, has appeared without proof in \cite[\S 3.4]{BW1}, and as we promised there,
we present the proof here. We also discuss %, briefly,
the problem of determining which
scalars appear as the determinant of a matrix of a given type. This problem appears intractable in general, and we provide only a very
partial answer.  But we identify a special case
of the problem which leads to a difficult but highly interesting combinatorial problem, to which we formulate Conjecture \ref{conj:sumsets} as
a plausible solution.

In \S 4 we make several observations concerning the problem of
commuting nilpotent classes;
this is a problem which has attracted attention in several different
contexts over the years,
and there is every reason to suppose that it is hard.
Among other results, we determine
in Theorem~\ref{thm:universal}
the nilpotent classes
which commute with every other
nilpotent class of the same dimension, %(Theorem~\ref{thm:universal}),
and in Theorem~\ref{thm:2partcommuting}
we classify all pairs of commuting nilpotent classes of matrices whose
nullities are at most $2$. % (Theorem~\ref{thm:2partcommuting}).
We describe a construction on matrices which produces interesting
and non-obvious examples of commuting nilpotent classes.
This construction motivates Theorem~\ref{thm:fielddependent},
which says that for every prime $p$
and positive integer~$r$, there exists a pair of classes of nilpotent
matrices which
commute over the field $\F_{p^a}$ if and only if $a>r$. As far as
the authors are
aware, it has not previously been observed that
the commuting of nilpotent classes, as parameterized by partitions,
is dependent on the field of definition.

More detailed outlines of the results of \S 2, \S 3 and \S 4
are to be found at the
beginnings of those sections.

\subsection{Background definitions}
We collect here the main
prerequisite definitions concerning partitions, classes and class types
that we require.

\subsubsection*{Partitions.}
We define a \emph{partition} to be a weakly decreasing sequence of finite length whose terms are positive integers;
these terms are called the \emph{parts} of the partition.
We shall denote the $j$-th part of a partition~$\lambda$
by $\lambda(j)$. The sum of the parts of $\lambda$ is written as $|\lambda|$.

Given partitions $\lambda$ and $\mu$, we write $\lambda + \mu$
for the partition of $|\lambda| + |\mu|$ whose multiset of
parts is the union of the multisets of parts of $\lambda$
and of $\mu$. We shall write
$2\lambda$ for $\lambda+\lambda$, and similarly we shall
define $t\lambda$
for all integers $t \in \N_0$.
A partition $\mu$ will be said to be \emph{$t$-divisible} if it is expressible as $t\lambda$ for some partition $\lambda$;
if $s\lambda = t\mu$ then we may write $\mu=\frac{s}{t}\lambda$.

We shall require the \emph{dominance order} $\unrhd$ on partitions.
For two partitions $\lambda$ and $\mu$ we say that $\lambda$ \emph{dominates} $\mu$, and write
$\lambda \unrhd \mu$ (or $\mu\unlhd\lambda$) if
\[ \sum_{i=1}^j \lambda(i) \ge \sum_{i=1}^j \mu(i) \]
for all $j \in \N$. (If $i$ exceeds the number
of parts in a partition, then the corresponding part %-size
is taken
to be $0$.)

Let $\lambda$ be a partition with largest part $\lambda(1) = a$.
The \emph{conjugate partition} $\overline{\lambda}$ is defined to be
$(\overline{\lambda}(1),\dots,\overline{\lambda}(a))$,
where $\overline{\lambda}(j)$ is the number of parts of $\lambda$ of size
at least $j$. It is a well-known fact (see for instance \cite[1.11]{Macdonald}) that the conjugation operation on partitions reverses the dominance order; that is,
$\lambda\unrhd \mu$ if and only if  %\Longleftrightarrow
$\overline{\mu}\unrhd\overline{\lambda}$.

A geometric interpretation of the dominance order is developed
by Gerstenhaber
in \cite{Gerstenhaber1} and \cite{Gerstenhaber2}; the
issues with which the latter paper is concerned are similar in many respects to those considered in \S 4 of the present paper,
although Gerstenhaber's approach using algebraic varieties is very different.

\subsubsection*{Similarity classes.}
Let $K$ %$\F$
be a field. A class of similar matrices in
$\Mat_n(K)$ %\F)$
is determined by the following data: a finite set $\mathcal{F}$ of
irreducible polynomials over $K$, %$\F$,
and for each
$f\in \mathcal{F}$ a partition $\lambda_f$ of a positive integer,
such that
\[
n=\sum_{f\in \mathcal{F}}|\lambda_f|\deg f.
\]
The characteristic polynomial of a matrix $M$ in this class is
$\prod_f f^{|\lambda_f|}$.
There is a decomposition of $V$ given by
\[
V=\bigoplus_f\bigoplus_jV_{f}(j),
\]
where $M$ acts indecomposably on the subspace $V_{f}(j)$
with characteristic polynomial $f^{\lambda_f(j)}$.
This decomposition is, in general, not unique. By a change of basis, we may express $M$ as
$\bigoplus_f\bigoplus_jP_f(j)$, where $P_f(j)$ is a matrix representing
 the action of $M$ on $V_f(j)$;
we say that~$P_f(j)$ is a \emph{cyclic block} of~$M$.

If $\mathcal{F}=\{f_1,\dots, f_t\}$ and the associated partitions are
$\lambda_1,\dots, \lambda_t$ respectively,
then we shall define the
\emph{cycle type} of $M$ to be the formal expression
\[
\cyc(M) = f_1^{\lambda_1}\cdots f_t^{\lambda_t}.
\]
The order in which the polynomials appear in this expression is, of course,
unimportant.

\subsubsection*{Nilpotent classes.}
We shall denote by $N(\lambda)$ the similarity class of nilpotent
matrices with cycle type
$f_0^{\lambda}$, where $f_0(x)=x$.
We denote by $J(\lambda)$ the unique matrix in upper-triangular
Jordan form in the similarity class $N(\lambda)$.

If $\lambda=(\lambda(1),\dots,\lambda(k))$ we shall omit
unnecessary brackets by writing
$N(\lambda(1),\dots,\lambda(k))$ for $N(\lambda)$
and $J(\lambda(1),\dots,\lambda(k))$ for $J(\lambda)$.

\subsubsection*{Class types.}
More general than the notion of similarity class is that of class type.
If $M$ is a matrix of cycle type $f_1^{\lambda_1}\cdots f_t^{\lambda_t}$,
where for each $i$ the polynomial $f_i$ has degree $d_i$,
then the \emph{class type} of $M$ is the formal string
\[
\ty(M) = d_1^{\lambda_1}\cdots d_t^{\lambda_t}.
\]
Here too, the order of the terms is unimportant.

Any string of this form will be called a \emph{type}.
The \emph{dimension} of the type $d_1^{\lambda_1}\cdots d_t^{\lambda_t}$
is defined to be $d_1 |\lambda_1| +  \cdots + d_t |\lambda_t|$.
We shall say that the type $T$ is \emph{representable} over a field $K$
if there exists a matrix of class type $T$ with entries in $K$;
the dimension of such a matrix
is the same as the dimension of the type.
Clearly not all types are representable over all fields;
for instance the type
$1^{\lambda}1^{\mu}1^{\nu}$ is not representable over $\F_2$
since there are
only two
distinct linear polynomials over this field; similarly $3^{\lambda}$ is not representable over $\R$ since there are no irreducible cubics
over $\R$.

Similar matrices have the same cycle type and the same class type, and so we may meaningfully attribute types of either kind to similarity classes.

We shall say that a class type $T$ is \emph{primary} if it is $d^{\lambda}$ for some~$d$ and~$\lambda$. Otherwise $T$ is \emph{compound}.
If $d^{\lambda}$ appears as a term in the type $T$, we say that $d^\lambda$ is a \emph{primary component} of~$T$. We may also say that a matrix, a similarity class of matrices,
or a cycle type is primary or compound, according to its class type, and we may refer to its primary components.

We have already defined what it means for two similarity
classes to commute. We generalise this idea to types, as follows.
\begin{definition}
Let $S$ and $T$ be class types. We say
that $S$ and $T$ \emph{commute} over a field~$K$ if
there are matrices $X$ and $Y$ over $K$
such that $X$ has class type $S$,
and $Y$ has class type $T$, and $X$ and $Y$ commute.
\end{definition}
The field $K$ will not always be mentioned explicitly
if it is clear from the context.

\section{Commuting types of matrices}\label{sect:reduction}
This section proceeds as follows.
In \S 2.1 we prove several results relating the class type of a polynomial
in a matrix $M$ to the class type of $M$, leading up to
Theorem \ref{thm:polytypes}: that two similarity classes have the same class type %got to here 02.18 14/12/09
if and only if they contain representatives which
are polynomial in one another. This
result is then used in the proof of Theorem~\ref{thm:typecommute},
which states that two similarity classes commute if and only
if their class types commute.

Using Theorem \ref{thm:typecommute}, we proceed to reduce our original
problem of deciding which similarity classes
commute, first to the case of primary types in \S 2.2, and thence to the case of nilpotent classes in \S 2.3.
At the end of \S 2.3 we give examples illustrating
both steps of this reduction.

\subsection{Polynomials and commuting types}

If $M$ is a matrix of primary class type $d^\lambda$
then it has associated with it a single irreducible
polynomial $f$ such that its cycle type is~$f^\lambda$.
It is clear that $f(M)$ is nilpotent. The following
lemma and proposition describe its associated partition.
\begin{lemma}\label{lemma:partsizecalculation}
Let $M$ be a matrix of cycle type $f^\lambda$, where $\deg f=d$.
For each $j$, let $m_j$ be the number of parts of $\lambda$
of size $j$. Then
\[
dm_j = (\nullity f(M)^j-\nullity f(M)^{j-1})-
(\nullity f(M)^{j+1}-\nullity f(M)^{j}).
\]
\end{lemma}
\begin{proof}
Let $P$ be a cyclic block of $M$. If the dimension of $P$
is $dh$ then the characteristic polynomial of $P$ is $f^h$.
If $j\ge h$, then $\nullity f(P)^j = dh$; otherwise $\nullity f(P)^j=dj$.

Since $M$ is a direct sum of cyclic blocks of dimensions $d\lambda(1),
d\lambda(2),\ldots$, it follows that
\[
\nullity  f(M)^j = \sum_{h\le j}dhm_h +\sum_{h>j}djm_h,
\]
and hence
\[
\nullity  f(M)^{j+1}-\nullity  f(M)^{j} = \sum_{h>j}dm_h.
\]
This implies the lemma.
\end{proof}

\begin{proposition}\label{prop:nilpotentpartsizes}
Let $M$ be a matrix of primary type $d^\lambda$. If the cycle
type of $M$ is $f^\lambda$ then $f(M)$ is nilpotent of type
$1^{d \lambda}$.
\end{proposition}
\begin{proof}
Since $f(M)$ is nilpotent, it is primary and its associated polynomial, $f_0(x)=x$, is linear. The result is now immediate from
Lemma~\ref{lemma:partsizecalculation}.
\end{proof}

We use the preceding proposition to give some information about
the type of $F(M)$, where~$M$ is a primary matrix
and $F$ is any polynomial. The following lemma will be required.

\begin{lemma}\label{lemma:dominance}
Let $M$ and $N$ be nilpotent matrices with associated partitions $\mu$ and $\nu$ respectively. Then $\mu\unlhd \nu$ if and only if
$\rank M^j\le \rank N^j$ for all $j\in\N$.
\end{lemma}
\begin{proof}
The rank of $M^j$ is equal to the sum of the $j$ smallest parts of the conjugate partition~$\overline{\mu}$. The rank of $N^j$ can be calculated
similarly in terms of $\overline{\nu}$. It follows easily that $\rank M^j\le \rank N^j$ for all $j$ if and only if $\overline{\mu}\unrhd
\overline{\nu}$. The lemma now follows from the fact that the dominance order $\unrhd$ is reversed by conjugation of partitions.
\end{proof}

\begin{proposition}\label{prop:polymatrix}
Let $X$ be a primary matrix of class type $d^{\lambda}$ with entries from a field~$K$,
and let $F \in K[x]$ be any polynomial.
The type of $F(X)$ is $e^\mu$ for some $e$ dividing $d$,
and some partition $\mu$ such that $e|\mu|=d|\lambda|$
and $e \mu \unlhd d \lambda$.
\end{proposition}
\begin{proof}
Let the cycle type of $X$ be $f^\lambda$ where
$f$ is an irreducible polynomial of degree $d$. If~$\alpha$ is
a root of $f$ in a splitting field, then the eigenvalues
of $F(X)$ are the conjugates over~$K$ of~$F(\alpha)$.
Hence $F(X)$ is of primary type, and
if $g \in K[x]$ is the irreducible polynomial associated with~$F(X)$,
then the degree of $g$ divides $d$. Let $e^\mu$ be the type of $F(X)$.

Let $Y = F(X)$.
We observe that  $g(Y) = (g \circ F)(X)$ is a nilpotent matrix,
and hence $f$ divides $g \circ F$; let $g \circ F = kf$.
By Proposition~\ref{prop:nilpotentpartsizes},
$f(X)$ has type $1^{d\lambda}$, while $g(Y)$ has type~$1^{e\mu}$.
For each $i \in \N_0$ we have
$g(Y)^i = k(X)^if(X)^i$
and hence
$\im g(Y)^i \subseteq \im f(X)^i$.
It follows that $\rank g(Y)^i \le \rank f(X)^i$
for every $i \in \N$. Now from Lemma \ref{lemma:dominance} we see that $e\mu \unlhd d\lambda$,
as required.
\end{proof}

When $K$ is a finite field, Proposition~\ref{prop:polymatrix} has the following partial converse.
\begin{proposition}\label{prop:typetopoly}
If $X$ is a primary matrix of class type $d^\lambda$ with entries from $\F_q$,
and $D$ is a similarity class of matrices also of this class type,
then there is a polynomial $F \in \F_q[x]$ such that $F(X) \in D$.
\end{proposition}
\begin{proof}
Let $f\in\F_q[x]$ be the irreducible polynomial associated with $X$.
Suppose that
the additive Jordan--Chevalley decomposition of $X$ is
$\overline{X}+N$, where $\overline{X}$ is semisimple and $N$ is nilpotent;
recall that $\overline{X}$ and $N$ can be expressed as
polynomials in $X$. Without loss of generality, we may suppose that
\[
\overline{X}=\diag(P,\dots,P),
\]
where the cyclic block $P$ has minimum polynomial $f$.

Let $g$ be the irreducible
polynomial associated with the similarity class $D$,
and
let $\alpha$ and $\beta$ be roots of $f$ and $g$ respectively in
$\F_{q^d}$. There exists a polynomial
$G \in \F_q[x]$, coprime with $f$, such that $G(\alpha)=\beta$.
If we define
\[
Q=G(P),
\]
then $Q$ has minimum polynomial $g$.
Let
\[
\overline{Y}=\diag(Q,\dots,Q).
\]
Then $\overline{Y}=G(\overline{X})$, and since $\overline{X}$ is
polynomial in $X$, it follows that
$\overline{Y}$ is too. Moreover, if we set $Y=\overline{Y}+N$, then $Y$ is polynomial in $X$, and it is clear that
$Y$ lies in the similarity class $D$.
\end{proof}

Let $C$ and $D$ be similarity classes of $\Mat_n(\F_q)$.
We say that $D$ is \emph{polynomial in $C$} if there exists
a polynomial $F$ with coefficients
in $\F_q$ such that $F(X)\in D$ for all $X\in C$.

\begin{theorem}\label{thm:polytypes}
Let $C$ and $D$ be similarity classes of $\Mat_n(\F_q)$.
The classes $C$ and $D$ have the same type if and only if
$C$ and $D$ are polynomial in one another.
\end{theorem}
\begin{proof}
We observe that applying a polynomial to a matrix cannot increase
its number of primary components. So if $C$ and $D$ are polynomial in one
another, then they have the same number of components.
Moreover there is a pairing between the
primary components $C_1,\dots,C_t$ of $C$ and
$D_1,\dots,D_t$ of $D$ such that $C_i$ and $D_i$ are polynomial in one another for all~$i$. It will therefore be sufficient to prove the result in the case
that both $C$ and $D$ are primary. %We therefore
Suppose that
$\ty(C) = d^\lambda$ for some $d \in \N$ and
some partition $\lambda$.
It follows from Proposition~\ref{prop:polymatrix} that $D$ has class
type $e^\mu$ where $e$ divides $d$ and $e \mu\unlhd d \lambda$.
By symmetry we see that $e=d$ and $\lambda = \mu$, as required.

For the converse, suppose that $\ty(C) = \ty(D)$.
Let $T_1,T_2,\dots, T_t$ be the primary components
of $\ty(C)$, and let
$X=\diag(X_1,\dots, X_t)$ be
an element of $C$ such that $\ty(X_i)=T_i$ for all $i$. Let the minimum polynomial of the block $X_i$
be $f_i^{a_i}$, where $f_i$ is irreducible.
By Proposition~\ref{prop:typetopoly},
there exist polynomials
$F_1,\dots,F_t \in \F_q[x]$ such that
$\diag(F_1(X_1),\dots,F_t(X_t))\in D$.
By the Chinese Remainder Theorem, there exists a polynomial $F\in\F_q[x]$ such that
\[
F(x)\equiv F_i(x) \bmod f_i^{a_i}(x) \ \ \textrm{for all } i.
\]
And now we see that $F(X)\in D$, as required.
\end{proof}

It was proved by Green
\cite[Lemma 2.1]{Green} that the type of a matrix determines its centralizer up to isomorphism. Using
Theorem~\ref{thm:polytypes} we may prove the following stronger result.
\begin{theorem}\label{thm:conjugatecentralizer}
Let $X$ and $Y$ be matrices in $\Mat_n(\F_q)$ with the same class type. Let $\Cent X$ and $\Cent Y$ be the centralizers in $\Mat_n(\F_q)$
of $X$ and $Y$ respectively. Then $\Cent X$ and $\Cent Y$ are conjugate by an element of $\GL_n(\F_q)$.
\end{theorem}
\begin{proof}
By Theorem \ref{thm:polytypes} there exist polynomials $F$ and $G$ such that $F(X)$ is conjugate to $Y$ and $G(Y)$ is conjugate to $X$.
Now the centralizer $\Cent F(X)$ is a subalgebra of $\Cent X$ which is conjugate to $\Cent Y$; similarly the centralizer $\Cent G(Y)$ is a subalgebra
of $\Cent Y$ which is conjugate to $\Cent X$. Since $\Cent X$ and $\Cent Y$ are finite, it is clear that $\Cent X=\Cent F(X)$ and that
$\Cent Y=\Cent G(Y)$, which suffices to prove the theorem.
\end{proof}

An obvious corollary of Theorem \ref{thm:polytypes}, which has been stated in \cite[\S 3.2]{BW2},
is that classes of the same type commute. We are now in a position to establish a stronger result. Recall that
types $S$ and $T$ are said to commute if there exist
commuting matrices $X$ and $Y$ with types $S$ and $T$ respectively.

\begin{theorem}\label{thm:typecommute}
Let $C$ and $D$ be similarity classes of matrices over $\F_q$. Then $C$ and $D$ commute if and only if
$\ty(C)$ and $\ty(D)$ commute.
\end{theorem}
\begin{proof}
One half of the double implication is trivial, since if the similarity classes commute then by definition
the class types do. For the other half, notice that
if $\ty(C)$ and $\ty(D)$ commute then there exist commuting similarity
classes $C'$ and $D'$ such that \mbox{$\ty(C')=\ty(C)$} and \mbox{$\ty(D')=\ty(D)$}. Let $X'$ and $Y'$ be commuting matrices
from $C'$ and $D'$ respectively. Then there exist polynomials $F$ and $G$ such that $F(X')\in C$ and $G(Y')\in D$, and
clearly $F(X')$ and $G(Y')$ commute.
\end{proof}

We remark that
Theorems~\ref{thm:polytypes},~\ref{thm:conjugatecentralizer} and~\ref{thm:typecommute}
do not hold for matrices over an arbitrary field. There are
counterexamples
in $\Mat_2(\Q)$, for instance. Let $\Q(\alpha)$ and $\Q(\beta)$ be distinct
quadratic extensions of $\Q$. Let $C$ and $D$ be the similarity classes of rational matrices with
characteristic polynomials $x^2-\alpha$ and $x^2-\beta$ respectively; then
$\ty(C)=\ty(D)=2^{(1)}$. Since the eigenvalues
$\alpha$ and $\beta$ are not polynomial in one another,
it is clear that neither are $C$ and $D$. Moreover, the classes $C$ and~$D$
do not commute. It is for this reason that our consideration of commuting types is
for the most part restricted to matrices with
entries from a finite field.

\subsection{Reduction to primary types}
The next step in our strategy is to reduce the question of which
class types commute
to the corresponding question about primary types.
This is accomplished
in Proposition~\ref{prop:doubledecomposition} below.

We shall need the following two definitions.
\begin{definition}
A \emph{separation operation} on a type $T$
is the replacement of a primary component $d^{\lambda}$ of $T$
by $d^{\mu}d^{\nu}$, where
$\lambda = \mu +\nu$. A \emph{separation} of $T$ is a type
obtained from $T$ by repeated applications of separation operations.
\end{definition}
\begin{definition}
Let $S$ and $T$ be types. We shall say that $S$ and $T$
\emph{commute componentwise} over a field $K$ if the primary components
of $S$ and $T$ can be ordered so that
$S=c_1^{\lambda_1}\cdots c_t^{\lambda_t}$ and
$T=d_1^{\mu_1}\cdots d_t^{\mu_k}$, where $c_i^{\lambda_i}$ commutes with $d_i^{\mu_i}$ over $K$
for each~$i$.
\end{definition}
This definition, it should be noted, does not preclude the possibility that types $S$ and $T$
commute componentwise, even if one or both of them
cannot be represented over the field~$K$.
For example, $1^{(1,1,1)}$ commutes componentwise
with $1^{(1)} 1^{(1)} 1^{(1)}$ over~$\F_2$ according to the definition, even though the latter type is not representable.
The examples at the end of \S 2.3 illustrate why this freedom is desirable.

\begin{proposition}\label{prop:doubledecomposition}
Let $S$ and $T$ be types which are representable over a finite field $\F_q$.
Then $S$ and $T$ commute over $\F_q$ if and only if there exist
separations $S^\star$ of $S$ and $T^\star$
of $T$ such that $S^\star$ and $T^\star$ commute componentwise.
\end{proposition}
\begin{proof} Let $X$ and $Y$ be commuting matrices with entries from $\F_q$, whose
types are $S$ and $T$ respectively. It is well known and
easy to show that there
exists a decomposition \mbox{$V= V_1\oplus\cdots\oplus V_t$} such that both
$X$ and $Y$ act as transformations of primary type on each of the
summands $V_i$. Suppose that the
action of $X$ on $V_i$ has type $c_i^{\lambda_i}$,
and the action of $Y$ has type $d_i^{\mu_i}$. Then it is clear that the
primary types $c_i^{\lambda_i}$ and $d_i^{\mu_i}$ commute, that
$c_1^{\lambda_1}\cdots c_t^{\lambda_t}$ is a separation of $S$
and that $d_1^{\mu_1}\cdots d_t^{\mu_t}$
is a separation of $T$.

For the converse, suppose that
the  primary types $c_i^{\lambda_i}$ and $d_i^{\mu_i}$ commute,
that $c_1^{\lambda_1}\cdots c_t^{\lambda_t}$ is a separation of $S$
and that $d_1^{\mu_1}\cdots d_t^{\mu_t}$ is a separation of~$T$.
Then, from Theorem~\ref{thm:typecommute}, it follows
that for any choice of
irreducible polynomials $f_i$ of degree $c_i$ and $g_i$ of
degree $d_i$, the classes $f_i^{\lambda_i}$ and $g_i^{\mu_i}$ commute.
If $X_i$ and $Y_i$ are commuting representatives of these respective
classes, then the matrices $X=\diag(X_1, \ldots, X_t)$
and $Y=\diag(Y_1, \ldots, Y_t)$ commute.
Now each primary type~$c_i^{\lambda_i}$ derives (under separation operations)
from a particular component of $S$. If we select our
polynomials $f_i$ in such a way that blocks deriving from the
same component of $S$ have the same polynomial, then we find
that $\ty(X)=S$. Similarly we can choose the
polynomials $g_i$ so that $\ty(Y)=T$, and it follows that $S$ and $T$ commute. \end{proof}

\subsection{Reduction to nilpotent classes}

We now complete
the reduction of our general problem of commuting classes to the case of nilpotent classes.
Recall that we denote by $N(\lambda)$ the similarity class of nilpotent
matrices with cycle type $f_0^{\lambda}$, where $f_0(x)=x$. Recall also that a partition is said to be $t$-divisible
if it is $t\nu$ for some partition $\nu$.

\begin{theorem}\label{thm:primarycommuting}
Let $S=c^{\lambda}$ and $T=d^{\mu}$ be primary types of the same
dimension. Let $h=\hcf(c,d)$ and $\ell=\lcm(c,d)$.
Then $S$ and $T$ commute over $\F_q$ if and only if
$\lambda$ is $\frac{d}{h}$-divisible, $\mu$ is $\frac{c}{h}$-divisible,
and the nilpotent classes $N(\frac{h}{d}\lambda)$ and $N(\frac{h}{c}\mu)$
commute over $\F_{q^{\ell}}$.
\end{theorem}
\begin{proof}
Suppose that $S$ and $T$ commute over $\F_q$.
Let $X$ and $Y$ be commuting elements of $\Mat_n(\F_q)$ with cycle types
$f^{\lambda}$ and $g^{\mu}$ respectively, where $\deg f=c$ and $\deg g=d$.
Let~$\alpha_1,\dots, \alpha_c$ be the roots of $f$ and
$\beta_1,\dots,\beta_d$ the roots of $g$ in the extension field
$\F_{q^\ell}$. Over this extension field,
it is easy to see that the cycle types of $X$ and $Y$ are given by
\begin{align*}
\cyc(X) &= (x-\alpha_1)^{\lambda}\cdots (x-\alpha_c)^{\lambda}, \\
\cyc(Y) &= (x-\beta_1)^{\mu}\cdots (x-\beta_d)^{\mu}.
\end{align*}

Let $W=\F_{q^\ell}^n$,
and let $W_{ij}$ denote the maximal
subspace of $W$ on which $X-\alpha_iI$ and $Y-\beta_jI$ are both nilpotent.
(So $W=\bigoplus_{ij}W_{ij}$.)
Let $\lambda_{ij}$ and~$\mu_{ij}$ be the partitions such that
the type of~$X$ on $W_{ij}$ is $1^{\lambda_{ij}}$
and the type of~$Y$ on $W_{ij}$ is $1^{\mu_{ij}}$.
Then clearly
$\sum_{j=1}^d \lambda_{ij} = \lambda$ for all $i$,
while $\sum_{i=1}^c \mu_{ij} = \mu$ for all~$j$.

Since $X$ and $Y$ have entries in $\F_q$, it follows that the
Frobenius automorphism $\xi\mapsto \xi^q$ of~$\F_{q^{\ell}}$ induces an isomorphism between the
$\F_{q^{\ell}}\langle X,Y\rangle$-modules $V_{ij}$ and $V_{i'j'}$
whenever \mbox{$i-j\equiv i'-j'\bmod h$}. Hence
\[
\lambda_{ij} = \lambda_{i'j'} \ \text{and} \ \mu_{ij} = \mu_{i'j'}
\ \text{whenever}\
i-j\equiv i'-j'\bmod h.
\]
%\[
%(\lambda_{ij},\mu_{ij}) = (\lambda_{i'j'},\mu_{i'j'}) \ \text{whenever}\
%i-j\equiv i'-j'\bmod h.
%\]
Therefore
the partitions $\lambda_{ij}$ for $i \in \{1,\ldots, c\}$ and
$j \in \{1,\ldots, d\}$ are determined
by the partitions
$\lambda_{1k}$ for $k \in \{1,\ldots, h\}$, and since
\[
\lambda=\frac{d}{h}\sum_{k=1}^h \lambda_{1k},
\]
it follows that $\lambda$ is $\frac{d}{h}$-divisible.
Similarly, $\mu$ is $\frac{c}{h}$-divisible.

Now clearly the actions of $X$ and $Y$ on the subspace
$\bigoplus_{k=1}^{h}V_{1k}$ commute.
The type of $X$ on this submodule (defined over $\F_{q^{\ell}}$)
is $1^{\lambda_{11}}\cdots 1^{\lambda_{1h}}$,
which is a separation of $1^{\frac{h}{d}\lambda}$.
Similarly the type of $Y$ on the submodule is a separation
of $1^{\frac{h}{c}\mu}$. Hence, by the `if' direction of
Proposition~\ref{prop:doubledecomposition},
the types $1^{\frac{h}{d}\lambda}$ and $1^{\frac{h}{c}\mu}$ commute
over $\F_{q^{\ell}}$.
In particular, it follows from
Theorem~\ref{thm:typecommute}
that the nilpotent classes $N(\frac{h}{d}\lambda)$ and $N(\frac{h}{c}\mu)$
 commute
over this field.

For the converse, let $\lambda'=\frac{h}{d}\lambda$ and $\mu'=\frac{h}{c}\mu$, and suppose that the nilpotent classes
$N(\lambda')$ and~$N(\mu')$ commute over $\F_{q^{\ell}}$.
We shall denote by $m$ the integer $|\lambda'|$, which of course is equal
to~$|\mu'|$.
Let $\alpha$ and $\beta$ be elements of $\F_{q^{\ell}}$ whose degrees
over $\F_q$ are $c$ and $d$ respectively.
Since $N(\lambda')$ and $N(\mu')$ commute over $\F_{q^{\ell}}$, so do the classes
with cycle types $(x-\alpha)^{\lambda'}$ and~\hbox{$(x-\beta)^{\mu'}$}.
Let~$X$ and $Y$ be commuting elements of these respective classes.
Let $\phi$ be an embedding of the matrix algebra $\Mat_m(\F_{q^\ell})$
into $\Mat_{\ell m}(\F_q)$; then it is not hard to see that~$\phi(X)$ has class type~$c^\lambda$
and $\phi(Y)$ has class type $d^\mu$. It follows that these types
commute over~$\F_q$.
\end{proof}

It is worth noting that Theorem~\ref{thm:fielddependent}
%Proposition~\ref{prop:nncommute}
below implies that the references to particular fields in the
statement of Theorem~\ref{thm:primarycommuting} are essential.
The following special case of the theorem, however, does not depend on the field of definition.

\begin{proposition}\label{prop:simpleprimarycommuting}
Let $d, k \in \N$. %Let $\lambda = (k)$ and let $\mu = d(k) = (k, \ldots, k)
The types $d^{(k)}$ and $1^{(k, \ldots, k)}$ commute over any field.
\end{proposition}

\begin{proof}
If the field in question is finite, then the proposition follows from
Theorem~\ref{thm:primarycommuting}. For it suffices to show that
the type $\frac{1}{d}(k, \ldots, k) = (k)$ commutes with
itself over
$\F_{q^d}$, and certainly this is the case.

A straightforward modification of the last paragraph of the proof of Theorem~\ref{thm:primarycommuting}
would allow us to deal with arbitrary fields;  however we prefer the following short argument involving tensor products.
Let $f$ be an irreducible polynomial of degree $d$ and let $P$ be the
companion matrix of $f$. The type $d^{(k)}$ is represented by the
$dk \times dk$ matrix
\[ P^{(k)} =
\left(\begin{matrix}
P & I &   &      \\
  & P & I &      \\
  &   &\ddots&\ddots\\
  &   &    &   P \end{matrix}\right). \]
Let $J = J(k)$ be the $k$-dimensional
Jordan block with eigenvalue $1$. It is clear that $P^{(k)}$
commutes with the tensor product $I \otimes J$ (which is obtained from the matrix above by substituting $I$ for each occurrence of $P$).
And $I \otimes J$ is conjugate to
$J\otimes I = \mathrm{diag}(J,\dots, J)$, which has type
$1^{(k,\dots,k)}$. Hence the types $d^{(k)}$ and $1^{(k, \ldots, k)}$ commute.
\end{proof}

We end this section with two examples of how the steps
in our reduction can be carried out, which illustrate
the various results of this section.
%In order to bring
%the examples to definite conclusions,
%we allow ourselves to use results on nilpotent
%classes which are proved in \S 4 below.

\begin{example}
Let $p$, $q$, $r$, $s$ and $t$ be the following irreducible polynomials 
over~$\F_2$:
\[
\begin{array}{ll}
\textrm{linear: } & p(x) = x,\  q(x) = x + 1;\\
\textrm{quadratic: } & r(x) = x^2 + x + 1;\\
\textrm{cubic: }  &  s(x) = x^3 + x + 1,\ t(x) = x^3 + x^2 + 1.
\end{array}
\]
Let $C$ be the similarity class of matrices
over $\F_2$ with cycle type $p^{(12,12)}q^{(2,2,2)}r^{(3)}s^{(1)}$
% \[ p^{(12,12)}q^{(2,2,2)}r^{(3)}s^{(1)} \]
and let~$D$ be the similarity class with cycle type $r^{(7,5)} t^{(2,2,1)}$.
%\[ r^{(7,5)} t^{(2,2,1)}. \]
We shall prove
that~$C$ commutes with~$D$.

By Theorem~\ref{thm:typecommute}, this
is equivalent
to showing that the %associated
types
\begin{align*}
S &= 1^{(12,12)}1^{(2,2,2)}2^{(3)}3^{(1)}, \\
T &= 2^{(7,5)}3^{(2,2,1)}
\end{align*}
commute. This, in turn, will follow from
Proposition~\ref{prop:doubledecomposition}, if we can show
that~$S$ commutes componentwise with the separation
$T^\star = 2^{(7,5)}3^{(2)}3^{(2)}3^{(1)}$ of $T$.
(This example was chosen to make the point that
it is not necessary that the separated types
can be represented over $\F_2$.)
By Theorem~\ref{thm:primarycommuting} we see that
$1^{(12,12)}$ commutes with $2^{(7,5)}$ over $\F_2$
if and only if $1^{(6,6)}$ commutes with $1^{(7,5)}$ over $\F_4$;
that this is the case follows from Proposition~\ref{prop:nncommute} below, which implies that the
nilpotent classes $N(6,6)$ and $N(7,5)$ commute over $\F_4$.
It is immediate from Theorem~\ref{thm:primarycommuting}
that $1^{(2,2,2)}$ commutes with $3^{(2)}$,
and that $2^{(3)}$ commutes with $3^{(2)}$.
Hence $S$ and $T^\star$ commute componentwise,
and so $C$ and $D$ commute.
\end{example}

The converse directions of
Proposition~\ref{prop:doubledecomposition} and
Theorem~\ref{thm:primarycommuting} can in principle be used as part of a argument that two similarity classes do not commute; again,
results about commuting of nilpotent classes will generally be needed to complete such an argument.
The following example is illustrative.
%, in which we must make use of another result
%from \S 4 below.
\begin{example}
Let the polynomials $p$, $q$, $r$, $s$ and $t$, the class $C$, and the type $S$ be as in the previous example.
Let $D$ be the similarity class over $\F_2$ with cycle type $r^{(8,4)}t^{(2,2,1)}$. The class type of $D$ is
\[
T=2^{(8,4)}3^{(2,2,1)}.
\]
Suppose that a separation $T^*$ of $T$ commutes componentwise with a separation $S^*$ of $S$; then one of $2^{(8)}$ or $2^{(8,4)}$ is a component
of $T^*$. The first possibility is ruled out since $S^*$ can have no component of dimension $16$. The only possible component of $S^*$ of dimension
$24$ is $1^{(12,12)}$, and so if our supposition is correct, then the primary types $2^{(8,4)}$ and $1^{(12,12)}$ must commute over~$\F_2$. By Theorem
\ref{thm:primarycommuting},
this is the case only if $1^{(8,4)}$ and $1^{(6,6)}$ commute. But
by Proposition~\ref{prop:twopart} below,
the nilpotent classes $N(8,4)$ and $N(6,6)$ do not commute over
any field. It follows that~$C$ and $D$ do not commute.
\end{example}

\section{Types and determinants}\label{sect:determinants}

The main object of this section is to establish Theorem \ref{thm:determinants},
concerning determinants of
elements of centralizer algebras. The following definition is key.

\begin{definition}
Let $M$ be a matrix with class type $d_1^{\lambda_1} \cdots d_t^{\lambda_t}$.
The \emph{part-size invariant} of~$M$ is defined to be the highest common factor
of all of the parts of the partitions~$\lambda_1, \ldots, \lambda_t$.
\end{definition}

\begin{theorem}\label{thm:determinants}
Let $M\in\Mat_n(\F_q)$ have part-size invariant $k$.
The determinants which occur in the centralizer of $M$ in
$\Mat_n(\F_q)$ are precisely the $k$-th powers in $\F_q$.
\end{theorem}

Part of the motivation for this investigation comes from
the authors' paper~\cite{BW1} on
the distribution of
conjugacy classes of a group $G$ across the cosets of a
normal subgroup $H$, where~$G/H$ is abelian.
The \emph{centralizing subgroup} of a class $C$ with respect to $H$ was
defined to be
the subgroup~$\Cent_G(g)\cdot H$, where
$g\in C$ may be chosen arbitrarily. It was proved that if $G$ is finite and $G/H$ is cyclic, then
the classes with
centralizing subgroup $K$ are uniformly distributed across the
cosets of $H$ in~$K$.

Theorem \ref{thm:determinants} treats the case where $G=\GL_n(\F_q)$ and
$H=\SL_n(\F_q)$.
It is clear that the subgroups $K$ lying in the range
$H\le K\le G$ may be defined in terms of the determinants of
their elements; specifically, the index
$|K:H|$ is equal to the order of the subgroup of $\F_q^{\times}$
generated by the determinants of the matrices in $K$.
Hence, in order to calculate the centralizing
subgroup of a matrix, we must decide which determinants
occur in its centralizer.
The following corollary of Theorem~\ref{thm:determinants}
shows that the answer to this question depends only on the
class type of the matrix concerned.

%The corollary for centralizing subgroups is as follows.
\begin{corollary}\label{cor:centsubgroup}
Let $M\in\GL_n(\F_q)$ have part-size invariant $k$, and let $c=\hcf(q-1,k)$.
The centralizing subgroup
of the conjugacy class of $M$
is the unique index $c$ subgroup of $\GL_n(\F_q)$ containing
$\SL_n(\F_q)$.
\end{corollary}

In \S \ref{subsec:detcent} below we prove a special case of Theorem \ref{thm:determinants}, namely that the determinants
in the centralizer of a nilpotent matrix are $k$-th powers, where $k$ is the part-size invariant.
The proof of Theorem~\ref{thm:determinants} is completed in \S \ref{subsec:detproof}.
We end in \S \ref{subsec:detclasses} by discussing the natural---but
surprisingly hard---question
of which scalars can appear as the determinant of
a matrix of a given type.

\subsection{Determinants in the centralizer of a nilpotent matrix}\label{subsec:detcent}
In this section we let
$M \in \Mat_n(\F_q)$ be a nilpotent matrix lying
in the similarity class $N(\lambda)$. Let
$A = \Cent M$ be the subalgebra of $\Mat_n(\F_q)$ consisting
of the matrices that centralize
$M$. We shall find the composition
factors of $V = \F_q^n$ as a right $A$-module; using this result we
describe the determinants of
the matrices of $A$.
For some related results on the lattice of $A$-submodules
of $V$, the reader is referred to \cite[Chapter 14]{Gohberg}.

\begin{definition}\label{defn:height}
For $v\in V$ we define the \emph{height} of $v$, written $\height(v)$,
to be the least integer~$h$ such that $v\in \ker M^h$.
\end{definition}

\begin{definition}\label{defn:cyclic}
We shall say that a vector $u \in V$ is a \emph{cyclic vector}
for $M$ if $u$ is not in the image of $M$.
\end{definition}

The proof of the following well-known lemma is straightforward, and
is omitted.

\begin{lemma}\label{lemma:cyclic}
An element $Y \in A$ is uniquely determined by its effect on
the cyclic vectors of $M$. If $u_1, \ldots, u_t$
are linearly independent cyclic vectors and $v_1, \ldots, v_t$
are any vectors such that $\height(v_i) \le \height(u_i)$ for every $i$,
then there is an element $Y \in A$ such that $u_iY = v_i$ for
each~$i$.
\end{lemma}

As in Lemma~\ref{lemma:partsizecalculation}, we
let $m_h$ be the number
of parts of $\lambda$ of size $h$.
For $h \in \N_0$, we shall write~$V_h$ for $\ker M^h$.

\begin{proposition}\label{prop:simples}
For each $h \in \N$,
the subspace $V_h$ is an $A$-submodule of~$V$
containing $V_{h+1}M + V_{h-1}$ as an
$A$-submodule.
Moreover if $m_h \not= 0$ then
\[ V_h \left/\right. (V_{h+1}M + V_{h-1}) \]
is a simple $A$-module of dimension~$m_h$.
\end{proposition}
\begin{proof}
The proof of the first statement is straightforward, and we omit it;
we shall outline a proof of the second statement.

Let $u_1, \ldots, u_{m_h}$ be
a maximal set of linearly independent cyclic vectors
each of height $h$. It is not hard to see that
$u_1, \ldots, u_{m_h}$ span a complement in $V_h$
to $V_{h+1}M + V_{h-1}$.
By the previous lemma, for any vectors $v_1, \ldots, v_{m_h}$ in
$V_h$,
there exists $Y \in A$
such that $u_iY = v_i$ for each~$i$.
This implies that $A$ acts as a full matrix
algebra in its action on the quotient module
$V_h / (V_{h+1}M + V_{h-1})$. Hence the quotient
module is simple.
\end{proof}

For $h$ such that $m_h \not =0$, let
$S_h = V_h / ( V_{h+1}M + V_{h-1})$
be the simple $A$-module constructed in Proposition \ref{prop:simples}.
If $h \not= h'$ and both
$S_h$ and $S_{h'}$ are defined, then by Lemma~\ref{lemma:cyclic}, it
is possible to define a matrix $Y \in A$
such that $Y$ acts as the identity
on the cyclic vectors spanning $S_h$,
and as the zero map on the cyclic vectors spanning $S_{h'}$.
The simple modules $S_h$
and $S_{h'}$ are therefore non-isomorphic as $A$-modules.

\begin{proposition}\label{prop:composeries}
The $A$-module $V$ has a composition series
in which the simple $A$-module~$S_h$ appears
with multiplicity $h$.
\end{proposition}
\begin{proof}
The action of the nilpotent matrix $M$ on $V_h$ induces a non-zero homomorphism
of simple $A$-modules
\[ \frac{ V_hM^{i-1}}{ V_{h+1}M^i +  V_{h-1}M^{i-1}}
\longrightarrow
\frac{V_hM^i}{V_{h+1}M^{i+1} + V_{h-1}M^i}
\]
for each $i$ such that $1 \le i \le h-1$.
This gives us $h$ distinct
composition factors of $V_h$, each isomorphic to $S_h$.
It now follows from the Jordan--H{\"o}lder theorem that
in any composition series of $V$,
the simple
module $S_h$ appears at least with multiplicity $h$. Finally,
by comparing dimensions using the equation
\[ \dim V = n = \sum_h h m_h = \sum_h h \dim S_h, \]
we see that equality holds for each $h$,
and that the $A$-module $V$ has no other composition factors.
\end{proof}

\begin{proposition}
\label{prop:nilpotentpowers}
If $M$ is nilpotent, and has part-size invariant $k$, then
the determinants that appear in $\Cent M$ are
$k$-th powers in $\F_q$.
\end{proposition}
\begin{proof}
Given $Y \in \Cent M$ let $Y_h$ denote the matrix in $\Mat_{m_h}(\F_q)$
which gives the action of~$Y$ on the simple $A$-module $S_h$.
Using the composition series given by the previous theorem
to compute $\det Y$ we get
\[ \det Y = \prod_{h \atop m_h \not= 0} (\det Y_h )^{h}. \]
Since the part-size invariant of $m$ is the highest common factor
of the set $\{h \mid m_h \not= 0 \}$, we see that $\det Y$
is a $k$-th power.
\end{proof}

It is worth remarking that it is also possible
to prove Proposition~\ref{prop:composeries}
in a way that gives the required composition series in an explicit
form. We have avoided this approach in order to keep the notation
as simple as possible. The following example indicates how
to construct a suitable basis of $V$ in a small case.
\begin{example}
Let  $M \in \Mat_5(\F_q)$ be a nilpotent matrix in the similarity
class $N(2,2,1)$. Let $u_1, u_2$ be cyclic vectors of $M$ of height $2$,
and let $v$ be a cyclic vector of $M$ of height $1$. Then
with respect to the basis $u_1, u_2, v, u_1M, u_2M$ of $\F_q^5$,
the centralizer of $M$ consists of all matrices of the form
\[ \left( \begin{matrix} \alpha & \beta & \star & \star & \star \\
                         \gamma & \delta & \star & \star & \star \\
                               &      & \zeta & \star & \star \\
                             &      &     & \alpha & \beta \\
                              &      &      & \gamma & \delta
                         \end{matrix}\right) \]
where gaps denote zero entries, and $\star$ is used to denote an entry we have no need to specify explicitly.
The key to obtaining this matrix in the required form is to order the elements of the %cyclic
basis
correctly. The following principles determine a suitable ordering on the basis: elements come
in decreasing order of height; cyclic vectors come first among elements of the same height;
if $b_i$ comes before $b_j$ then $b_iM$ comes before $b_jM$.
\end{example}

\subsection{Proof of Theorem~\ref{thm:determinants}}\label{subsec:detproof}

The proof has two steps. We first show
that if $M$ is a matrix with entries
in $\F_q$ and part-size invariant $k$,
then every $k$-th power in $\F_q$ appears
as the determinant of a matrix in $\Cent M$.
In the second, we use Proposition~\ref{prop:nilpotentpowers}
to show that no other powers can appear.

We begin with the following lemma.

\begin{lemma}\label{lemma:polynomialconstant}
Let $S(k)$ be the set of $k$-th powers in $\F_q$. Let $d\in\N$ and $\theta\in\F_q^{\times}$. Then the number of
irreducible polynomials of degree $d$ over $\F_q$ with
constant term $\theta$ is
\[
\frac{1}{d(q-1)}\sum_{k|d \atop S(k)\ni\theta}\mu(k)\hcf(q-1,k)(q^{d/k}-1).
\]
This number is non-zero for all choices of $d$ and $\theta$ and for all $q$.
\end{lemma}
\begin{proof}
We give an elementary proof of the existence of a polynomial with degree $d$ and constant term $\theta$. For the
number of polynomials, see for instance \cite[\S 5.2]{Britnell}.

Let $\alpha$ be a generator of the multiplicative group $\F_{q^d}^\times$, and let $\beta=\mathfrak{n}(\alpha)$
where $\mathfrak{n} : \F_{q^d}^{\times} \rightarrow \F_q^{\times}$ is the norm homomorphism. It is clear that
$\beta$ generates $\F_q^\times$. Let $c$ be such that $0<c<q$ and $(-1)^d\theta=\beta^c$. Since $\F_{q^d}$ has
no proper subfield of index less than $q$, and since the multiplicative order of $\alpha^c$ is at least
$(q^d-1)/c$, it is easy to see that $\alpha^c$ cannot lie in a proper subfield of~$\F_{q^d}$. It follows that the
minimum polynomial of $\alpha^c$ over $\F_q$ has degree $d$ and constant term $\theta$, as required.
\end{proof}

\begin{proposition}\label{prop:cyclicpowers}
Let $P$ be a matrix with class type $d^{(j)}$. Then for any $\theta\in\F_q$, there exists a matrix in $\Cent P$
with determinant $\theta^j$.
\end{proposition}
\begin{proof}
We may assume that $\theta$ is non-zero. By Lemma \ref{lemma:polynomialconstant} there exists an irreducible
polynomial $f$ over $\F_q$ with degree $d$ and constant term $(-1)^d\theta$. Let $C$ be the similarity class
containing $P$, and let $D$ be the class of matrices with cycle type $f^{(j)}$. Since $C$ and $D$ have the same
class type, it follows from Theorem \ref{thm:polytypes} that they commute. Therefore $P$ commutes with an
element of $D$. It is clear from the construction of $D$ that its elements have determinant $\theta^j$, as
required.
\end{proof}

We now extend Proposition~\ref{prop:cyclicpowers} to a general matrix.

%It will be helpful to define the
%\emph{height} of a primary matrix $M$ to be the size of the largest part of the associated partition. (If $M$ acts nilpotently
%on a space $V$ then the height of $M$
%is the maximum height of a vector in $V$, as defined in \S 3.1.)

\begin{proposition}
\label{prop:existenceofpowers} If $M$ is a matrix with part-size invariant $k$,
then for any $\zeta\in\F_q$,
there exists a matrix in $\Cent M$ with determinant $\zeta^k$.
\end{proposition}

\begin{proof}
Let $P_1,\dots, P_s$ be the distinct cyclic blocks of $M$; so $M$ is conjugate to $\bigoplus_iP_i$. For each
$i$ let the class type of the block $B_i$ be $d_i^{h_i}$. By Proposition~\ref{prop:cyclicpowers}, for any
scalars~$\theta_i$ that we choose, there exist matrices $X_1,\dots,X_s$ such that $X_i\in\Cent B_i $ for all
$i$, and $\det X_i=\theta_i^{h_i}$. Thus~$M$ commutes with a conjugate of the matrix $\diag(X_1,\dots,X_s)$,
which has determinant~$\prod_i \theta_i^{h_i}$.
%$\theta_1^{h_1} \cdots \theta_s^{h_s}$.

It will therefore be enough to show that there exist non-zero scalars $\theta_1,\dots,\theta_s$ such that~$\prod_i \theta_i^{h_i}=\zeta^k$. But we know that $k=\hcf(h_1,\dots, h_s)$, and so there exist integers $a_i$
such that $k=\sum_ia_ih_i$; it follows that we can simply take $\theta_i=\zeta^{a_i}$ for all $i$.
\end{proof}

We now turn to the second step in the proof of Theorem~\ref{thm:determinants}.
\begin{proposition}\label{prop:powers}
Let $M$ be a matrix with part-size invariant $k$. The determinant
of an element of $\Cent M$ is a $k$-th power in $\F_q$.
\end{proposition}
\begin{proof} Let $M$ act on $V = \F_q^n$.
For each irreducible polynomial $f$ over $\F_q$ which divides the minimal polynomial of $M$, let $V_f$ be the
largest subspace of $V$ on which $f(M)$ acts nilpotently. Then $V=\bigoplus V_f$, and each summand $V_f$ is
invariant under $\Cent M$. It follows that if $Y\in\Cent M$ then $\det Y=\prod \det Y_f$, where $Y_f$ is the
restriction of $Y$ to $V_f$. Therefore, it will be sufficient to show that $\det Y$ is a $k$-th power for each
$f$.

Let $\lambda=(h_1,\dots, h_s)$ be the partition associated with a given $f$ in the rational canonical form of
$M$. From the definition of the part-size invariant, each of the parts $h_i$ is divisible by $k$. Let~$M_f$ be
the restriction of $M$ to $V_f$, and let $Y_f\in\Cent M_f$.

By Proposition \ref{prop:nilpotentpartsizes}, $f(M_f)$ is nilpotent with associated partition $d\lambda$, where
$d$ is the degree of $f$. It is clear, then, that the part-size invariant of $f(M_f)$ is %divisible by
$k$. Since
$Y_f$ is in the centralizer of $f(M_f)$, it follows from Proposition \ref{prop:nilpotentpowers} that $\det Y_f$
is a $k$-th power in $\F_q$, as required.
\end{proof}

Combining the results of Propositions~\ref{prop:existenceofpowers} and~\ref{prop:powers} gives
Theorem~\ref{thm:determinants}. %$\hfill\proofbox$

\subsection{Determinants in classes of a given type}\label{subsec:detclasses}

It is natural to ask which determinants are represented among matrices of a given type. This question leads to a
hard problem in arithmetic combinatorics, to which we have been able to find only a partial solution.

It is clear that if $T$ is a type representable over the field $\F_q$, then there is a matrix of type~$T$ with
zero determinant if and only if $T$ has a primary component $1^\lambda$ for some $\lambda$. This leaves us to
decide which non-zero determinants can arise. For primary types this question is easily answered.
\begin{lemma}\label{lemma:detprimary}
Let $\lambda$ be a partition of $k \in \N$, let $d\in\N$, and let $\theta\in\F_q^{\times}$. There is an
invertible matrix over $\F_q$ with type $d^\lambda$ and determinant $\theta$ if and only if $\theta$ is a $k$-th
power in $\F_q^\times$.
\end{lemma}
\begin{proof}
If $M$ is a matrix of type $d^\lambda$ then $M$ has characteristic polynomial $f^k$. The determinant of $M$ is
therefore a $k$-th power. That every $k$-th power in $\F_q^\times$ is obtained in this way follows easily from
Lemma~\ref{lemma:polynomialconstant}.
\end{proof}

The following pair of propositions establish
a sufficient condition on a type for it to represent all non-zero determinants.

\begin{proposition}\label{prop:typedeterminants}
Let $d \in \N$ be coprime with $q-1$, and let $T = d^{\lambda_1} \cdots d^{\lambda_t}$
be a type representable over $\F_q$.
If $L = |\lambda_1| + \cdots + |\lambda_t|$ is also coprime
to $q-1$, then every element of~$\F_q^\times$ is the determinant
of a matrix of type $T$.
\end{proposition}

\begin{proof}
It is an easy consequence of Lemma \ref{lemma:polynomialconstant} that if $d$ is coprime with $q-1$, then there
are the same number of irreducible polynomials of degree $d$ with any non-zero constant term. It follows that,
for a generator $\theta$ of the cyclic group $\F_q^{\times}$, there exists a permutation $\sigma$ of the set of
irreducible polynomials of degree $d$, such that $f^{\sigma}(0)=\theta f(0)$ for all $f$.

Let $C$ be a similarity class of type $T$, whose members have determinant $\alpha$. Consider the class $C'$
obtained from $C$ by applying the permutation $\sigma$ to the irreducible polynomials which appear in its cycle
type. It is easy to see that $C'$ has the same type as $C$, and that the members of $C'$ have determinant
$\alpha\theta^L$, where $L$ is as in the statement of the proposition. Now $\theta^L$ is a generator of
$\F_q^{\times}$ since $L$ is coprime with $q-1$, and so it is clear that by repeated applications of the
permutation $\sigma$ we can obtain any non-zero determinant of our choice.
\end{proof}

\begin{proposition}\label{prop:generaltypedeterminants}
Let $T$ be a type representable
over a finite field $\F_q$.
For each $d$ let~$L_d$ be the sum of the sizes of the partitions associated with the
components of degree $d$ in $T$. If~$dL_d$ is coprime with $q-1$ for any $d$, then every element of $\F_q^\times$ is a determinant of
a matrix of type $T$.
\end{proposition}

\begin{proof} This follows immediately from Proposition \ref{prop:typedeterminants}.
\end{proof}
It should be noted that Proposition \ref{prop:generaltypedeterminants} does not come close to giving a necessary condition for a
type to contain all non-zero determinants. Finding conditions which are both necessary and sufficient appears to be a highly intractable problem.

A special case of considerable interest is that of \emph{linear} types, of the form $1^{\lambda_1} \cdots
1^{\lambda_t}$. (These are precisely the types of triangular matrices over $\F_q$.) We make use of the following
definition.
\begin{definition}
Let $A$ be an abelian group of order $m$ (written multiplicatively) and let
\hbox{$\pi = (\pi_1, \ldots, \pi_m)
\in \Z^m$}. We say that an element $x \in A$ is \emph{$\pi$-expressible}
if there exists an ordering  $g_1, \ldots, g_m$ of the elements of $G$ such that
$x = g_1^{\pi_1} \cdots g_m^{\pi_m}$.
\end{definition}

The relevance of this definition to our problem is easily explained.
Let $T$ be the linear type $1^{\lambda_1}\cdots 1^{\lambda_t}$ where $t\le q-1$. Let $\pi\in \Z^{q-1}$ be
defined by
\[
\pi=(|\lambda_1|,\dots,|\lambda_t|,0,\dots,0).
\]
Then we observe that the non-zero determinants
represented in $T$ are precisely the $\pi$-expressible elements of $\F_q^{\times}$.

If $A$ is an abelian group of exponent $n$ then
we observe that adding multiples of $n$ to the entries of
$\pi$ does not affect $\pi$-expressibility in $A$; we may therefore assume that all of the entries of $\pi$
satisfy $0\le \pi_i\le n-1$. Similarly, reordering the entries of $\pi$ cannot affect $\pi$-expressibility, and
so we may suppose that they appear in decreasing order.

Numerical evidence obtained by the authors supports the following conjecture.

\vbox{
\begin{conjecture}\label{conj:sumsets}
Let $A$ be a cyclic group of order $m$. Let $\pi = (\pi_1, \ldots, \pi_m) \in (\Z / m\Z)^m$, where $\pi_1\ge
\cdots \ge \pi_m$. Let $\pi'$ be the partition obtained from $\pi$ by subtracting $\pi_m$ from each part
(thereby ensuring that the last part is $0$). Then every element of $A$ is $\pi$-expressible unless one of the following
holds:
\begin{enumerate}
\item $\pi' = (m-r,r,0,\ldots,0)$ for some $r$, or \item There exists an integer $p>1$ which divides each part
of $\pi'$, and which also divides $m$.
\end{enumerate}
\end{conjecture}}

This conjecture is known to be true in the case that $m$ is a prime (see
 \cite[Theorem~1.2]{Gacs}). For our purposes, we
would like it to be true for $A=\F_q^{\times}$ for all $q$; that is, whenever $m+1$ is a power of a prime.
This would provide a complete classification of the determinants occurring in linear types. In the very special case when $q= 2^r$ and $|\F_q^\times| = 2^r-1$
is a Mersenne prime, the result of \cite{Gacs} already gives
such a classification.

%prime, the results of \cite{Gacs} gives such a classification.

\section{Commuting nilpotent classes}\label{sect:nilpotent}

%Add remark about parabolic subalgebra agreeing with centralizer
%if and only if AR-refinement?

In \S\ref{sect:reduction} the question of which
similarity classes of matrices
over a finite field commute was reduced to the analogous problem
for nilpotent classes.
The question of which nilpotent classes
commute with a given nilpotent class $N(\lambda)$ appears to be a very hard problem, and we shall not attempt to answer it in any generality. We shall, however, treat a variety of special cases,
and make a number of observations which, so far as we have been able to determine, do not appear in the existing literature.
Our approach is elementary, and leads to results
which, for the most part, apply to matrices defined over
an arbitrary field.
 (For
some other recent results on the problem of commuting nilpotent
classes over algebraically closed fields, obtained
by the methods of Lie theory,
the reader is referred to \cite{Oblak} and~\cite{Panyushev}.)

Our results may be summarized as follows. Proposition \ref{prop:OnePart}
describes the nilpotent classes
that commute with $N(\lambda)$ when $\lambda$ has a single part.
This result has appeared previously in \cite{Oblak}; our Proposition \ref{prop:ARrefcommuting} is similar to, but slightly stronger than, the result
which appears there as Proposition 2.

Similarly, we deal in Proposition \ref{prop:n-1Part} with the case that
$\lambda=(n-1,1)$ for some $n$,  and in Proposition \ref{prop:all2s} with the case that $\lambda=(2,\dots,2)$.
Using these results we are able to classify those nilpotent classes that commute with every nilpotent class of the same dimension; this
is Theorem \ref{thm:universal}.

We next establish a condition for the nilpotent classes $N(n,n)$ and $N(n+1,n-1)$ to commute; these classes are found to commute
over any infinite field, and over the finite field~$\F_{p^r}$ provided that $p(p^{2r}-1)/e$ does not divide $n$, where $e=1$ if $p=2$ and
$e=2$ otherwise. As well as augmenting our list of commuting classes, this result is particularly significant, since it demonstrates
that commuting of classes is in some cases dependent on the field of definition. Finally, we use the results just mentioned to classify those
 commuting nilpotent classes whose associated
partitions have no more than two parts; this result, stated as Theorem \ref{thm:2partcommuting}, is valid over any field.

The following definition will be useful in what follows.

\begin{definition}
Let $M$ be a nilpotent transformation of a space $V$. A \emph{cyclic basis} for $M$ is a basis $B$ of $V$ with the property that for each $v\in B$,
either $vM=0$, or else $vM\in B$.
\end{definition}

Earlier in \S 3.1 we defined a cyclic vector
for $M$ to be a vector which is not in the image of~$M$.
Let $M\in N(h_1,\dots,h_k)$,
and let $B$ be a cyclic basis for $M$. Then $B$ contains cyclic vectors $v_1,\dots,v_k$, where $\height v_i=h_i$ for all $i$; in fact
%we have
\[
B=\{v_iM^j\mid 1\le i\le k, 0\le j< h_i\}.
\]
By Lemma \ref{lemma:cyclic}, an element of $\Cent M$ is determined by its action on $v_1,\dots, v_k$.

\subsection{Cyclic nilpotent classes and partition refinements}

Recall that $J(\lambda)$, or $J(\lambda(1),\dots,\lambda(k))$,
is the unique upper-triangular matrix in Jordan form in the similarity
class~$N(\lambda)$.
The next proposition is concerned with the case where
$\lambda=(n)$ for some $n$. It is well known that the elements
of the centralizer algebra $\Cent J(n)$ are the polynomials in $J(n)$---see for example \cite[Ch.\ III, Corollary to Theorem 17]{Jacobson}.

\begin{proposition}\label{prop:OnePart} Let $\lambda=(h_1,\dots,h_k)$. Then
$J(n)$ commutes with a conjugate of $J(\lambda)$ if and only if $h_1-h_k\le 1$.
%(So $\lambda$ has parts of at most $2$ different sizes,
%differing by at most $1$.)
\end{proposition}
\begin{proof} Write $E_i$ for the matrix whose $(x,y)$-th entry is $1$
if $k=y-x$, and $0$ otherwise. The matrices
$E_0,E_1,\dots, E_{n-1}$ form a basis for the centralizer algebra of $J(n)$. Let $M$ be a non-zero nilpotent element
of this algebra; then for some $d$ in the range $0< d\le n-1$ we can write
\[
M=\sum_{i\ge d} \alpha_i E_i,
\]
for scalars $\alpha_i$, with $\alpha_d\neq 0$.

It is easy to check that $\nullity  M^s =  \min(sd,n)$
for all integers $s$. Let $h$ be the least integer such that $hd\ge n$. Then it follows from Lemma~\ref{lemma:partsizecalculation} that
$M$ is conjugate to $J(\lambda)$, where
\[ \lambda=(h,\dots,h,h-1,\dots,h-1) \]
is the partition with $n-hd$ parts of size $h-1$
and $(h+1)d-n$ parts of size~$h$. This establishes the proposition.
\end{proof}

The terminology in the first
of the following definitions is borrowed from \hbox{\cite[\S 3]{KosirOblak}}.
\begin{definition}
A partition is \emph{almost rectangular} if its largest part differs from its smallest part by at most~$1$.
\end{definition}

\begin{definition}
Let $\lambda$ and $\mu$ be partitions.
We say that $\mu$ is a \emph{refinement} of $\lambda$ if $\mu$ is the disjoint union of subpartitions
whose sizes are the parts of $\lambda$.
We say that a refinement of $\lambda$ is
\emph{almost rectangular} if all of the subpartitions involved are almost rectangular.
\end{definition}

For instance, $(5,3,1)=(3+2, 2+1, 1)$ has $(3,2,2,1,1)$
as an almost-rectangular
refinement.
It is worth noting that while the relation given by ``$\mu$ is a refinement of $\lambda$'' is clearly transitive,
the relation given by ``$\mu$ is an almost rectangular refinement of
$\lambda$'' is not.

\begin{proposition}\label{prop:ARrefcommuting}
Let $\mu_1$ and $\mu_2$ be partitions of $n$. If there exists a partition $\lambda$ which has
both $\mu_1$ and $\mu_2$ as almost rectangular refinements, then the conjugacy classes represented by the Jordan blocks $J(\mu_1)$ and
$J(\mu_2)$ commute.
\end{proposition}
\begin{proof} Consider the subpartitions $\nu_1$ of $\mu_1$ and $\nu_2$ of $\mu_2$ whose parts combine to create a single part
of $\lambda$ of size $h$. Since $\nu_1$ and $\nu_2$ are almost rectangular, they yield Jordan blocks whose classes
commute with that of $J(h)$.
But the centralizer of $J(h)$ consists of polynomials in~$J(h)$,
and it follows that the classes of $J(\nu_1)$
and $J(\nu_2)$ have representatives which are polynomials in~$J(h)$.
So these representatives commute, and hence
$J(\mu_1)$ and $J(\mu_2)$ have conjugates which commute. \end{proof}

The preceding proposition
is slightly more general than \cite[Proposition~2]{Oblak}, which states that the nilpotent classes $N(\lambda)$ and $N(\mu)$
commute if $\mu$ is an almost rectangular refinement of~$\lambda$.
It is noted in \cite{Oblak} that there exist examples of classes commuting that cannot be
explained in this way. We remark that our Proposition
\ref{prop:ARrefcommuting} does not account for all commuting between
classes,
either. We illustrate this
fact with the example and the proposition below; other
examples will be seen in subsequent sections.

\vbox{
\begin{example}
There is no partition which has
both $(2,2)$ and $(3,1)$ as an almost rectangular refinement, but the classes $N(2,2)$ and $N(3,1)$ commute over any field.
We leave the proof of this to the reader, while remarking
that it is a special case of any one of
Propositions \ref{prop:n-1Part},~\ref{prop:all2s}
and~\ref{prop:nncommute} below.
\end{example}}

\begin{proposition}
Let $\lambda$ be a partition, and let $\overline{\lambda}$ be its conjugate partition. Then the nilpotent classes with partitions $\lambda$
and $\overline{\lambda}$ commute.
\end{proposition}

\begin{proof}
Let $\lambda=(h_1,\dots,h_k)$,
where $h_1\ge \cdots\ge h_k$. Let $N$ be nilpotent of type $\lambda$, and let
$u_1, \dots, u_k$ be
cyclic vectors for $N$, such that $u_i$ has height $h_i$ for all $i$. By Lemma \ref{lemma:cyclic} there is a unique
matrix $M\in\Cent N$ such that $u_iM = u_{i+1}$
for all $i$, with $u_kM=0$.

If $\lambda=(5,5,3,2)$, for instance, then the actions of $N$ and $M$ on the cyclic basis can be represented
as follows:

\vspace{-20pt}
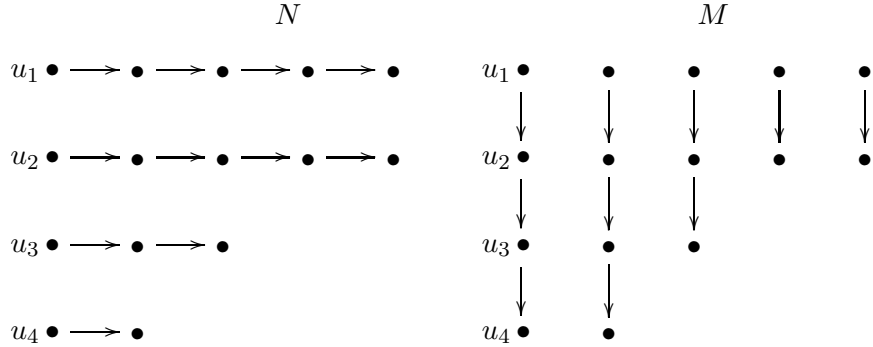
\begin{align*} N \hspace{2.05in} M \hspace{0.8in} \\
\xymatrix@C=18pt@R=18pt{
u_1\, \bullet \ar@<1pt>[r] & \bullet \ar@<1pt>[r] & \bullet \ar@<1pt>[r]
& \bullet \ar@<1pt>[r] & \bullet \\
u_2\, \bullet \ar@<1pt>[r] & \bullet \ar@<1pt>[r] & \bullet \ar@<1pt>[r] & \bullet \ar@<1pt>[r] & \bullet \\
u_3\, \bullet \ar@<1pt>[r] & \bullet \ar@<1pt>[r] & \bullet  \\
u_4\, \bullet \ar@<1pt>[r] & \bullet
}
\qquad
\xymatrix@C=18pt@R=18pt{
u_1\, \bullet \ar@<5.5pt>[d] & \bullet \ar[d] & \bullet \ar[d] & \bullet \ar[d] & \bullet \ar[d] \\
u_2\, \bullet \ar@<5.5pt>[d] & \bullet \ar[d] & \bullet \ar[d] & \bullet & \bullet \\
u_3\, \bullet \ar@<5.5pt>[d] & \bullet \ar[d] & \bullet  \\
u_4\, \bullet & \bullet
}
\end{align*}

\smallskip

\noindent It is easy to check that $M$ is nilpotent, with associated partition $\overline{\lambda}$.\end{proof}%

In general there does not
exist a partition which has both $\lambda$ and $\overline{\lambda}$ as almost
rectangular refinements, as is shown by the example illustrating the proof above,
or by the case $\lambda = (4,1,1)$.

\subsection{Universally commuting classes}

The object of this section is to classify, in Theorem~\ref{thm:universal},
the partitions to which the following definition refers.
\begin{definition}
A partition $\lambda$ of $n$ is \emph{universal} with respect to a field $K$ if
$N(\lambda)$ commutes with $N(\mu)$ over $K$ for every partition $\mu$ of $n$.
\end{definition}

The reference to the field in this definition is in fact redundant;
it is a consequence of Theorem~\ref{thm:universal} that
a partition which is universal with respect to one field is universal with respect to any field.
To prove the theorem, we shall require the following two propositions.

\vbox{
\begin{proposition}\label{prop:n-1Part}
Let $\lambda$ be a partition of $n$. The matrix
$J(n-1,1)$ commutes with a conjugate of $J(\lambda)$ if and only if one of the following holds:
\begin{enumerate}
\item $\lambda$ has a part of size $1$, and if $\lambda^-$ is obtained from $\lambda$ by removing this part, then $J(n-1)$ commutes with a
conjugate of $J(\lambda^-)$; Proposition \ref{prop:OnePart} provides a classification in this case.
\item $n$ is even, and all of the parts of $\lambda$ are of size $2$.
\item $\lambda$ has a part of size $3$, and its other parts are of size $1$ or $2$, with at least one part of size $1$.
\item $n=3$ and $\lambda=(3)$.
\end{enumerate}
\end{proposition}}

\begin{proof}
The centralizer algebra of $J(n-1,1)$ has the basis
\[
\{ E_i\mid 0\le i\le n-2 \} \cup \{F,G,H\},
\]
where
\[
\sum_i \alpha_i E_i+\beta F+\gamma G +\delta H=\left(\begin{array}{cccccc}
\alpha_0    & \alpha_1 & \alpha_2 & \dots  & \alpha_{n-2} & \beta \\
0      & \alpha_0 & \alpha_1 &        & \alpha_{n-3} & 0 \\
0      & 0   & \alpha_0 &        & \alpha_{n-4} & 0 \\
\vdots &     &     & \ddots &         & \vdots \\
0      & 0   & 0   &        & \alpha_0     & 0 \\
0      & 0   & 0   & \dots  & \gamma       & \delta
\end{array}\right).
\]
A nilpotent element of this algebra must have $\alpha _0= \delta =0$.
We suppose that $M$ is such an element, and that $M$ is non-zero.
By Lemma~\ref{lemma:partsizecalculation} the partition $\lambda$
associated with $M$ is determined by the sequence of ranks
of powers of $M$.

If $\alpha_i=0$ for all $i < n-2$,
then $\alpha_{n-2}, \beta$ and $\gamma$ are the only entries that are possibly non-zero.
It is easy to see that
the rank sequence 
\[ (\rank I, \rank M, \rank M^2, \rank M^3) \] 
must be either
$(n,2,1,0)$ or $(n,1,0,0)$. In the first case
the partition associated with $M$ is
$(3,1^{n-3})$, which is covered by either part (iii) or part (iv) of the lemma.
In the second case the partition is $(2,1^{n-2})$,
which is covered by part (i) or part (ii).

Now suppose that there exists $i<n-2$ such that $\alpha_i\neq 0$.
Let $m$ be the least such $i$. If \mbox{$m< (n-2)/2$}
then it is not hard to see that
the rank sequence is
\[
(n,n-m-1,n-2m-1,\dots, 0).
\]
The partition $\lambda$ given by this data has one more
part of size $1$ than the
partition $\lambda^-$ given by the data
\[
(n-1,n-m-1,n-2m-1,\dots,0).
\]
But~$\lambda^-$
corresponds to the rank sequence for an element of the centralizer algebra
of $J(n-1)$, and so this case is covered by case (i) of the lemma. If
$m>(n-2)/2$ then the same situation occurs if $\beta \gamma=0$.
But if $\beta$ and $\gamma$ are both
non-zero then the rank sequence obtained is \mbox{$(n-m-1,1,0)$}. The corresponding
partition $\lambda$ is covered by part
(iii) of the lemma.

The final case to analyse occurs when $n$ is even and $m=(n-2)/2$.
If $M^2 \not= 0$ then the situation of the previous paragraph applies.
Otherwise the rank sequence is $(n,n/2,0)$ and all of
the parts of $\lambda$ have size $2$, as in part (ii) of the lemma.
\end{proof}

\begin{proposition}\label{prop:all2s}
Let $\lambda$ be the partition of $2s$ which has
$s$ parts of size $2$, and let $\mu$ be any partition of $2s$. Then $J(\lambda)$ commutes with a conjugate of
$J(\mu)$.
\end{proposition}
\begin{proof}
By a straightforward inductive argument, we may suppose that $\mu$ has no subpartition of even size. If $\mu$ has only one part then the result follows
from Proposition \ref{prop:OnePart}; so we may assume that $\mu$ has exactly two parts, $s+t$ and $s-t$.

A cyclic basis for $N=J(\lambda)$ has the form $B=\{e_1,\dots, e_s,f_1,\dots,f_s\}$, where the vectors $f_i$ are in the kernel of $N$, and
$e_iN=f_i$ for all $i$. Let $M$ be the matrix whose action is defined by
$e_i M = e_{i+1}$, $f_i M = f_{i+1}$ for $1 \le i < s$, and
\begin{align*}
e_s M &= \begin{cases} f_{s-t+1} & \text{if $t > 0$,} \\
                       0         & \text{otherwise,} \end{cases} \\
f_s M &= 0.
\end{align*}
It is easy to see that $M$ commutes with $N$, hence it suffices
to show that $M \in N(\mu)$.
A basis for $\ker M$ is given by $\{f_s, e_s-f_{s-t}\}$, so
$\nullity M = 2$.
It follows that the partition
associated with~$M$ has two parts, and since $e_1$ is a cyclic vector of height $s+t$,
this partition must be~$(s+t,s-t)$, as required.
\end{proof}

\vbox{
\begin{theorem}\label{thm:universal}
The universal partitions are precisely those with no part greater than $2$, together with $\lambda=(3)$.
\end{theorem}
\begin{proof} Suppose that $\lambda$ has no part of size greater than $2$.
If all of the parts of $\lambda$ have size $2$, then~$J(\lambda)$ commutes with all nilpotent classes, by Proposition \ref{prop:all2s}.
Otherwise $\lambda$ has a subpartition $\lambda_m$ of $m$ for every $m\le n$. Let $\mu$ be a partition of $n$ with largest part $m$. Then
since $\lambda_m$ is an almost rectangular refinement of $m$, it follows from Proposition \ref{prop:OnePart} that $J(\lambda_m)$
commutes with a conjugate of $J(m)$. Now if $\lambda'$ denotes the partition obtained by deleting the parts of $\lambda_m$ from $\lambda$,
and if $\mu'$ is obtained by deleting a
part of size $m$ from $\mu$, then we may suppose inductively that $J(\lambda')$ commutes with a conjugate of $J(\mu')$.
It follows that $J(\lambda)$ commutes with a conjugate of $J(\mu)$.

Conversely, suppose that $\lambda$ has largest part $h>2$. If $J(\lambda)$ commutes with $J(n)$ then by Proposition \ref{prop:OnePart} all of
its parts have size $h$ or $h-1$. Then we see from Proposition \ref{prop:n-1Part} that $J(\lambda)$ does not commute with a conjugate of $J(n-1,1)$,
except in the single case that $\lambda=(3)$. \end{proof}}

\subsection{Commuting of classes $N(n,n)$ and $N(n+1,n-1)$}

The main object of this section is to prove Proposition~\ref{prop:nncommute}
below, which gives a necessary and sufficient condition for the
classes $N(n,n)$ and $N(n+1,n-1)$ to commute.
This case
is of particular interest because the field enters in an essential way.
In Theorem~\ref{thm:fielddependent} we use this proposition to show
that  for every prime $p$ and positive integer $r$, there exists a pair of
classes of nilpotent matrices which commute over the field $\F_{p^r}$
if and only if $s > r$.

Proposition~\ref{prop:nncommute} is motivated by a natural construction on matrices.
Suppose that $X$ and $Y$ are commuting matrices over a field $K$, and let
\[
D = \left( \begin{matrix} X & 0 \\ 0 & X \end{matrix} \right),
   \qquad
   E = \left( \begin{matrix} Y & I \\ 0 & Y \end{matrix} \right).
\]
Clearly the matrices $D$ and $E$ commute. We may assume that $X$ and $Y$ (and hence~$D$ and~$E$) are nilpotent; then this
construction (and other similar ones) may in principal be used to find new cases of commuting nilpotent classes. The partition labelling the class of $D$
is clearly $2\lambda$, where $\lambda$ labels the class of $X$. The partition labelling the class of $E$ is harder to calculate, and depends on the characteristic of $K$.

We have no occasion
to make systematic use of this construction in the present paper, but the following
example is illustrative. Let $X=Y=J(n)$. Then $D \in N(n,n)$. The partition labelling
the class of $E$ is $(n+1,n-1)$ except in the case that $\chr K$ divides $n$, in which case it is $(n,n)$. It follows that $N(n,n)$ and
$N(n+1,n-1)$ commute over fields of all but finitely many characteristics, the exceptions being the prime divisors of $n$. We note, however,
that the present method gives no information about whether the classes commute in fields of these exceptional characteristics; this gives an indication
that the following proposition is non-trivial.

\begin{proposition}\label{prop:nncommute}
Let $p$ be a prime, and let
\[
e=\left\{\begin{array}{ll} 1 & \textrm{if $p=2$,}\\ 2 &\textrm{otherwise.}\end{array}\right.
\]
Then the nilpotent types $(n,n)$ and $(n+1,n-1)$ commute over $\F_{p^r}$ if and only if
$n$ is not divisible by $p(p^{2r}-1)/e$.
\end{proposition}

\begin{proof} %Let $p$ be a prime.
Let $M$ be nilpotent of type $(n+1,n-1)$, acting on a space~$V$ over $\F_{p^r}$. Take a cyclic basis
$\{u_i,w_j\mid 0\le i\le n,\ 1\le j\le n-1\}$ for $V$,
with $u_iM=u_{i-1}$ and $w_jM=w_{j-1}$ for all $i$ and~$j$. Let $U_k$ and $W_k$ denote the subspaces
$\langle u_j\mid 0\le j\le k\rangle$ and $\langle w_j\mid 1\le j\le k\rangle$ respectively---we take $W_0=\{0\}$ and $W_n=W_{n-1}$. Let $V_k$
denote $U_{k}\oplus W_k$ for all $k$. For each pair $(x,y)$ with $x\in V_{n-1}$ and $y\in V_{n-2}$, there is an unique nilpotent element $Y$ of
$\Cent M$ such that $u_nY=x$ and $w_{n-1}Y=y$; it follows
from Lemma~\ref{lemma:cyclic} that
all of the nilpotent elements of $\Cent M$ can be obtained in this way.

Let $Y\in \Cent M$ be nilpotent, and define $\alpha,\beta,\gamma,\delta$ by
\begin{eqnarray*}
u_{n}Y & \in & \alpha u_{n-1}+\gamma w_{n-1}+ V_{n-2}, \\
w_{n-1}Y & \in & \beta u_{n-2}+\delta w_{n-2}+ V_{n-3}.
\end{eqnarray*}
The reader may find helpful the following diagrammatic
representation of~$Y$.

\vspace{-6pt}

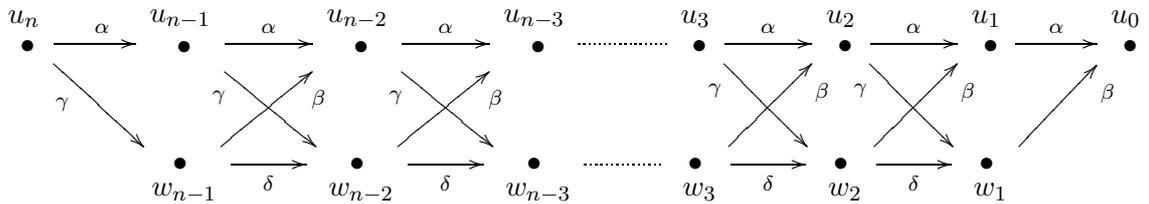
\[
\hspace{-0.425in}
\newcommand{\tl}[1]{\displaystyle #1 \atop \displaystyle \phantom{\mid} \bullet}
\newcommand{\bl}[1]{\displaystyle \bullet \atop \displaystyle \phantom{\mid} #1}
\xymatrix@R=28pt@C=28pt{
\tl{u_n} \ar@<-3.5pt>[r]^{\alpha}\ar[dr]_<<<<{\gamma} &
\tl{u_{n-1}} \ar@<-3.5pt>[r]^{\alpha}\ar[dr]_<<{\gamma} &
\tl{u_{n-2}} \ar@<-3.5pt>[r]^{\alpha}\ar[dr]_<<{\gamma} &
\tl{u_{n-3}} \ar@<-3.5pt>@{..}[r] &
\tl{u_3} \ar@<-3.5pt>[r]^{\alpha}\ar[dr]_<<{\gamma} &
\tl{u_2} \ar@<-3.5pt>[r]^{\alpha}\ar[dr]_<<{\gamma} &
\tl{u_1} \ar@<-3.5pt>[r]^{\alpha} &
\tl{u_0} & \\
&
\bl{w_{n-1}} \ar@<6pt>[r]_\delta\ar[ur]_>>\beta &
\bl{w_{n-2}} \ar@<6pt>[r]_\delta\ar[ur]_>>\beta &
\bl{w_{n-3}} \ar@<6pt>@{..}[r] &
\bl{w_3} \ar@<6pt>[r]_\delta\ar[ur]_>>\beta &
\bl{w_2} \ar@<6pt>[r]_\delta\ar[ur]_>>\beta &
\bl{w_1} \ar[ur]_>>\beta
}
\]

\smallskip

The matrix $A=\left(\begin{array}{cc}\alpha&\beta\\ \gamma&\delta\end{array}\right)$ describes the maps induced by $Y$,
\[
\overline{Y}_k:\ \frac{V_{k}}{V_{k-1}}\ \longrightarrow\ \frac{V_{k-1}}{V_{k-2}},
\]
where $k$ is in the range $1<k<n$. Outside of this range, the map $\overline{Y}_n$ has domain $\langle\ u_n+V_{n-1}\ \rangle$ of dimension $1$,
while $\overline{Y}_1$
has codomain $\langle u_0\rangle$ of dimension $1$. These maps are represented by the first row and the first column of $A$ respectively.

\vbox{
\begin{claim} The kernel of $Y$ has dimension $2$ if and only if $A$ is non-singular.
\end{claim}
%{\noindent\it Proof of claim.}
\begin{proof}[Proof of Claim]
If $A$ is invertible, then the maps $\overline{Y}_k$ are injective for $k>1$.
It follows easily that if $v\in\ker Y$ then $v\in V_1$.
It is now easy to check that $v \in \langle u_1, \beta u_2-\alpha w_1 \rangle$ and so $\nullity Y=2$ in this case.

Conversely, suppose that $A$ is singular. If $\alpha=\beta=0$ then $V_1\subseteq\ker Y$, and so $\nullity Y\ge 3$. So let us suppose that
$\alpha$ and
$\beta$ are not both $0$. Then there
exists $z\in V_1$ such that $zY=u_0$. Since~$A$ is
singular, the map $\overline{Y}_2$ has
a non-trivial kernel, and it follows that there exists $v\in V_2\setminus V_1$ such that $vY\in V_0=\langle u_0\rangle$. Say that
$vY=\sigma u_0$; now we have a set of three kernel vectors, $\{u_0,\ \beta u_1-\alpha w_1,\ v-\sigma z\}$,
which is linearly independent since~$v-\sigma z\notin V_1$. So $\nullity Y\ge 3$ in this case as well.\end{proof}%\hfill\proofbox
}

The dimension of $\ker Y$ tells us the number of parts in the partition associated with the class of $Y$. This partition therefore has two parts if
and only if the matrix $A$ is non-singular. Note that since $Y^{n+1}=0$, no part can be larger than $n+1$, and therefore the only possible partitions
are $(n+1,n-1)$
and $(n,n)$. The former corresponds to the class of $M$ itself, while the
latter case occurs when $Y^n=0$, which is the case if and only if $u_n\in \ker Y^n$.

Now we observe that
\begin{eqnarray*}
u_nY^n &=& \overline{Y}_1\circ \overline{Y}_2\circ\cdots\circ\overline{Y}_n (v_n+V_{n-1})\\
&=& u_0 R_1 A^{n-2} C_1 \;,
\end{eqnarray*}
where $R_1$ and $C_1$ are, respectively, the first row and the first column of $A$. So the partition of~$Y$ is $(n,n)$ precisely
when~$R_1A^{n-2}C_1 = (0)$, or equivalently, when the matrix $A^n$ has a zero for its top left-hand entry.

\begin{claim}
Every element of $\GL_2(\F_{p^r})$ is
either a scalar matrix, or else is conjugate to a matrix
with a zero for its top left-hand entry.
\end{claim}
%{\noindent\it Proof of claim.}
\begin{proof}[Proof of Claim]Every quadratic polynomial over $\F_{p^r}$ is the characteristic polynomial of a
unique similarity class of non-scalar matrices. Thus if $X$
is a non-scalar matrix with characteristic polynomial $x^2+\sigma x+\tau$, then $X$ is conjugate to
\[
\scriptstyle \left(\scriptstyle\begin{array}{cc} 0& 1\\  -\tau&-\sigma\end{array}\right),
\]
 as required.\end{proof}%$\hfill\proofbox$

Now suppose that $\GL_2(\F_{p^r})$ contains a non-scalar element $X$ which is an $n$-th power in the group.
Then $X$ has
a conjugate $X'$ with a zero for its top left-hand entry. Clearly $X'$ is also an $n$-th power; by choosing $a,b,c,\delta$ to be the entries of an
$n$-th root of $X'$, we can construct a matrix~$Y$ in $\Cent M$ whose type is $(n,n)$.

There exist non-scalar $n$-th powers in $\GL_2(\F_{p^r})$ provided
that $n$ is not divisible by the exponent of $\PGL_2(\F_{p^r})$. This exponent is
$p(p^{2r}-1)/e$, and the proof of Proposition \ref{prop:nncommute} is complete.
\end{proof}

\begin{remark}
This argument also goes to show that the nilpotent types $(n,n)$ and
\hbox{$(n-1,n+1)$} commute over any infinite field $K$, since the
exponent of $\PGL_2(K)$ is infinite.
\end{remark}

\begin{theorem}\label{thm:fielddependent}
Let $p$ be a prime, and $r\ge 1$. There exist partitions $\lambda$ and $\mu$, such that~$N(\lambda)$ commutes with $N(\mu)$ over the fields
$\F_{p^{a}}$ for $a>r$, but not for $a\le r$.
\end{theorem}
\begin{proof} We use a famous theorem of Zsigmondy \cite{Zsigmondy} which states that if $k\ge 2$, $t\ge 3$, and $(k,t)\neq (2,6)$, then there is a
prime divisor of $k^t-1$ which does not divide $k^s-1$ for any $s$
such that $1 \le s<t$.

Let $L=\lcm(\{p^{2s}-1 \mid 1\le s\le r\})$, and let $n=pL/e$. %Then
We observe that $p(p^{2a}-1)/e$ divides~$n$ whenever $a\le r$. When $a>r$
we invoke Zsigmondy's Theorem with $(k,t)=(p,2a)$, or with $(k,t)=(4,3)$ if $p=2$ and $t=3$; this tells us
that $p^{2a}-1$ has a prime divisor $q$ which does not divide $p^{2s}-1$ for $s<a$. Clearly $q$ does not divide $n$, and so $p(p^{2a}-1)/e$
does not divide~$n$. It now follows from Proposition \ref{prop:nncommute} that the partitions $(n,n)$ and $(n+1,n-1)$ have the property stated in the
theorem.
\end{proof}

\begin{remark}
 The authors have found no case where the commuting of nilpotent classes depends on the field of definition in dimension less than $12$.
 This is the dimension of the smallest example given by Proposition \ref{prop:nncommute}: that of
 $N(6,6)$ and $N(7,5)$, which commute over every field except $\F_2$.
\end{remark}

\subsection{Classes corresponding to two-part partitions}

We end by establishing a result which, together with results already presented, will allow us to classify, over any field $K$,
pairs of partitions $(\lambda,\mu)$ with at most two
parts, such that~$N(\lambda)$ and $N(\mu)$ commute over $K$. We note that classes with at most $2$ parts are precisely those
whose elements have nullity at most $2$.

\vbox{
\begin{proposition}\label{prop:twopart}
Let $\lambda=(a,b)$ and $\mu=(c,d)$, where $a+b=c+d$ and $a> c\ge d>b$. If $N(\lambda)$ and $N(\mu)$ commute over a field $K$
then $c=d$ and $a-b=2$.
\end{proposition}
\begin{proof}
The case that $c=d$ and $a-b=2$ has been dealt with in Proposition \ref{prop:nncommute} and the ensuing remark.
We may therefore suppose that $a-b>2$. Let $M\in N(\lambda)$, and let $\{v,vM,\dots,vM^{a-1},w,wM,\dots,wM^{b-1}\}$ be a cyclic basis for $M$.
Let $W=\ker M^{a-2}$; so $W$ is the span of all the basis vectors apart from $v$ and $vM$. Suppose that $Y$ is nilpotent and commutes with $M$;
then it is not hard to see that $W\subseteq \ker Y^{a-2}$. Since $Y$ is nilpotent we have
$vY\in\alpha vM +W$ for some $\alpha\in K$.

Suppose first that $\alpha\neq 0$; then we see that $vY^{a-1}= \alpha^{a-1}vM^{a-1}$, while
$vY^a=0$. Hence $v$ is a cyclic vector for $Y$ of height $a$. It follows that if the partition associated with $Y$ has only $2$ parts then it
must be $\lambda$.

Suppose alternatively that $\alpha=0$, so $vY\in W$.
We shall show
that $\nullity Y \ge 3$, and so the partition associated with $Y$ has more than $2$ parts.
First observe that $vM^{a-2}$ and $vM^{a-1}$ are in $\ker Y$, since $vM^{a-2}Y=vYM^{a-2}\in WM^{a-2}=\{0\}$. Furthermore it is easy to show that
$vM^{a-3}Y$ and $wM^{b-1}Y$ both lie in $\langle vM^{a-1}\rangle$, and hence a non-zero linear combination of these two vectors lies in
$\ker Y$. We have therefore found three linearly independent vectors in $\ker Y$, as required.\end{proof}}

The following theorem simply collects together elements of Propositions \ref{prop:OnePart}, \ref{prop:nncommute} and \ref{prop:twopart}; it
requires no further proof.

\vbox{
\begin{theorem}\label{thm:2partcommuting}
Suppose that $\lambda$ and $\mu$ are partitions of $n$ with at most two parts, and that $N(\lambda)$ and $N(\mu)$ commute over a field $K$.
Assume without loss of generality that %$\lambda(1) \ge \mu(1)$.
the largest part of $\lambda$ is at least as large
as the largest part of $\mu$.
Then one of the following holds.
\begin{enumerate}
\item $\lambda=\mu$.
\item $n=2m$, $\lambda=(n)$ and $\mu=(m,m)$.
\item $n=2m$, $\lambda=(m+1,m-1)$, $\mu=(m,m)$ and, if $K$ is finite then the exponent of $\PGL_2(K)$ does not divide $m$.
\item $n=2m+1$, $\lambda=(n)$ and $\mu=(m+1,m)$.
\end{enumerate}
\end{theorem}
}

%\affiliationone{School of Mathematics \\
%University of Bristol\\
%University Walk \\
%Bristol BS8 1TW \\ United Kingdom
%\email{j.r.britnell@bristol.ac.uk\\mark.wildon@bristol.ac.uk}}

\end{document}